\documentclass{article}

\usepackage{a4wide,enumerate,amsmath,amssymb,amsthm,url,comment}

\newtheorem{theorem}{Theorem}
\newtheorem{lemma}[theorem]{Lemma}

\newtheorem{conjecture}[theorem]{Conjecture}

\newtheorem{definition}[theorem]{Definition}
\newtheorem{proposition}[theorem]{Proposition}

\newtheorem{remark}[theorem]{Remark}

\newtheorem{fha}[theorem]{Fundamental Heuristic Assumption}

\newcommand{\N}{\mathbb{N}}
\newcommand{\Z}{\mathbb{Z}}
\newcommand{\Q}{\mathbb{Q}}
\newcommand{\F}{\mathbb{F}}
\newcommand{\cF}{\mathcal{F}}
\DeclareMathOperator{\SL}{SL}

\newcommand{\sltz}{\SL_2(\Z)}
\DeclareMathOperator{\Ind}{Ind}
\DeclareMathOperator{\Gal}{Gal}
\DeclareMathOperator{\GL}{GL}
\DeclareMathOperator{\CL}{CL}
\DeclareMathOperator{\Aut}{Aut}
\newcommand{\Magma}{{\sc Magma}}

\newcommand{\mat}[4]{
 \left(  \begin{smallmatrix} #1 & #2 \\ #3 & #4 \end{smallmatrix} \right)}
\newcommand{\natden}[2]{\Delta(#1,#2)}

\newcommand{\ppp}[1]{\pi^{(#1)}}
\newcommand{\pp}[1]{\pi_{(#1)}}
\newcommand{\Ab}[1]{\underline{\mathrm{Ab}}_{(#1)}}
\newcommand{\Abp}[1]{\underline{\mathrm{Ab}}^{(#1)}}

\title{On mod~$p$ representations which are defined over~$\F_p$: II}

\author{L.\ J.\ P.\ Kilford and Gabor Wiese}

\begin{document}

\maketitle

\begin{abstract}
The behaviour of Hecke polynomials modulo~$p$ has been the subject of 
some study. In this note we show that, if~$p$ is a prime, the set of integers~$N$ such that the Hecke polynomials~$T^{N,\chi}_{\ell,k}$ for all primes~$\ell$, all weights~$k \ge 2$ and all characters~$\chi$ taking values in~$\{\pm 1\}$ splits completely modulo~$p$ has density~0, unconditionally for~$p=2$ and under the Cohen-Lenstra heuristics for~$p \ge 3$.
The method of proof is based on the construction of suitable dihedral modular forms.

\end{abstract}

\section{Introduction}

Let~$N$ and~$k$ be positive integers and let~$\ell$ and~$p$ be prime numbers. We will let~$S_k(\Gamma_0(N),\chi)$ be the space of holomorphic cusp forms of integer weight~$k$ for the congruence subgroup~$\Gamma_0(N)$ and the Dirichlet character~$\chi$ of modulus~$N$, and we will define~$T_{\ell,k}^{N,\chi}$ to be the characteristic polynomial of the Hecke operator~$T_\ell$ acting on~$S_k(\Gamma_0(N),\chi)$. We will call this polynomial the \emph{Hecke polynomial}.

We recall that for modular forms in characteristic~0, there is a 
well-known conjecture (Maeda's conjecture) that says that the 
characteristic polynomials of the Hecke operators acting on modular forms 
for the full modular group~$\sltz$ are irreducible: 
\begin{conjecture}[Maeda's Conjecture]
Let~$k$ be a positive integer, and let~$\ell$ be a prime number.
The Hecke polynomial~$T^{1,1}_{\ell,k} \in \Z[X]$ is irreducible, with Galois group~$S_n$, where~$n$ is the dimension of~$S_k(\sltz)$ as a complex vector space.
\end{conjecture}
This conjecture lends itself to numerical verification. Methods introduced in~\cite{buzzard-t2} prove that certain Hecke polynomials are irreducible and have full Galois group, and results such as those in~\cite{conrey-farmer-wallace}, \cite{baba-murty} and~\cite{ahlgren-new} show that if a certain~$T^{1,1}_{\ell,k}$ is irreducible then other~$T^{1,1}_{r,k}$ must be irreducible also. 

In the characteristic~$p$ case, however, things are obviously different. The paper~\cite{kilford-mod-p}, using methods developed 
in~\cite{conrey-farmer-wallace}, gives a list of spaces of modular forms for which one can prove that all of the Hecke 
polynomials~$T^{N,1}_{\ell,k}$ split into linear factors modulo~$p$. It is then asked whether these are all such spaces. In this paper we will give at least a partial answer to that question, which depends for odd primes~$p$ on the Cohen-Lenstra heuristics on class groups of imaginary quadratic fields.

\begin{theorem}
\label{main-theorem}
Let $p$ be a prime; if $p \ge 3$, assume the Cohen-Lenstra heuristics. Then the set of integers~$N$ such that the Hecke polynomials~$T^{N,\chi}_{\ell,k}$ for all primes~$\ell$, all weights $k \ge 2$ and
all characters~$\chi: (\Z/N\Z)^\times \to \{\pm 1\} \subseteq \F_p^\times$
split completely modulo~$p$ has density~0.
\end{theorem}

It should be pointed out that for given~$p$ and given level~$N$ the weight of an eigenform over~$\overline{\F}_p$ can always be adjusted to lie between $2$ and $p^2 - 1$.

It should also be noted that a natural generalization of Maeda's conjecture in characteristic~0 to congruence subgroups cannot be true in general; in~\cite{koo-2007} a Hecke eigenform of weight~$2$ on $\Gamma_0(63)$ is exhibited such that for a set of primes~$\ell$ of positive density the characteristic polynomial of the Hecke operator $T_\ell$ acting on the span of all the Galois conjugates of the form is reducible. In other words, the $\ell$-th coefficient does not generate the whole coefficient field for a set of primes~$\ell$ of positive density. This phenomenon is due to the existence of a nontrivial inner twist. Even in the absence of nontrivial inner twists, numerical evidence suggests that there exist examples where the set of such~$\ell$ is still infinite, although of density~$0$.

Theorem~\ref{main-theorem} will be proved in Section~\ref{sec:htwo} for $p=2$ and
in Section~\ref{sec:hp} for odd~$p$. It will be derived from a statement
on the class groups of imaginary quadratic fields which implies the existence of
dihedral modular forms mod~$p$ whose coefficient fields are not the prime field~$\F_p$
(see Section~\ref{sec:dih}).

One can imagine other ways for constructing mod~$p$ eigenforms with
$q$-expansions not in~$\F_p$. For instance, one could use families of
hyperelliptic curves of genus greater than~$1$ whose Jacobians are of
$\GL_2$-type in order to treat those primes~$p$ that have a nontrivial residue
degree in the endomorphism algebra of the Jacobian tensored with~$\Q$ (which
is the coefficient field of the corresponding holomorphic eigenform).
However, it does not seem obvious how to obtain the desired density~$1$
statement on the levels.

Moreover, techniques of level raising etc., as they are for instance used
in~\cite{dieulefait-wiese}, also easily yield mod~$p$ Hecke eigenforms having a
nontrivial coefficient field. However, the levels will always contain at
least a square, excluding a density~$1$ statement.

\section{Dihedral Galois representations and Hypothesis~$H_p$}\label{sec:dih}

\begin{definition}
Let $p$ be a prime number.
An abelian group~$C$ is called \emph{$p$-suitable} if $G$ has a cyclic quotient of order~$h$
such that $p \nmid h$ and $h \nmid (p^2-1)$.
\end{definition}

The condition of $p$-suitability is equivalent to the existence of
a cyclic quotient $H$ of~$G$ such that $H$ is isomorphic to
a subgroup of $\overline{\F}_p^\times$ but not to a subgroup of $\F_{p^2}^\times$.

We now prove the following results about $p$-suitable groups.

\begin{proposition}\label{prop:dihedral}
Let $G$ be a $p$-suitable abelian group and let $H$ be a cyclic quotient
of~$G$ of order~$h$ such that $H \subset \overline{\F}_p^\times$ without being
isomorphic to a subgroup of $\F_{p^2}^\times$.
Then the group
$$D = \langle \mat 0110, \mat x00{x^{-1}} \; \mid \; x \in H \rangle \subseteq \GL_2(\overline{\F}_p)$$
is isomorphic to the dihedral group $D_h$ of order $2h$ and not all the traces of elements
of~$D$ lie in~$\F_p$.
\end{proposition}

\begin{proof}
Let $x \in H$.
One just needs to observe that conjugation by $\mat 0110$ maps $\mat x00{x^{-1}}$
to $\mat {x^{-1}}00x$ in order to see that $D$ is actually isomorphic to~$D_h$.
Suppose that $x+x^{-1} = a \in \F_p$. Then $x$ is a root of the polynomial $X^2-aX + 1 \in \F_p[X]$
and consequently $x \in \F_{p^2}$. The assumption excludes that this happens for all~$x \in H$.
\end{proof}

\begin{proposition}\label{prop:rep}
Let $K/\Q$ be an imaginary quadratic field of discriminant~$d$ and let $C$ be its class group.
If $C$ is $p$-suitable,
then there exists an irreducible odd dihedral Galois representation
$\rho: \Gal(\overline{\Q}/\Q) \to \GL_2(\overline{\F}_p)$ of conductor~$| d |$ such that not all
its traces lie in~$\F_p$.
\end{proposition}

\begin{proof}
Let $H$ and~$h$ be as in Proposition~\ref{prop:dihedral}.
Note that $\Gal(K/\Q)$ acts on~$C$ and, hence, also on~$H$ by inversion.
We now view $H$ as an unramified character
$\chi: \Gal(\overline{\Q}/K) \to \overline{\F}^\times$ of order~$h$.
We choose $\sigma$ to be a lift to $\Gal(\overline{\Q}/\Q)$ of the nontrivial
element of~$\Gal(K/\Q)$. Then we have that
$\chi(\sigma \tau \sigma^{-1}) = \chi(\tau^{-1}) = \chi(\tau)^{-1}$ for any $\tau \in \Gal(\overline{\Q}/K)$.
This means that we can identify the group $D$ from Proposition~\ref{prop:dihedral}
in a natural way with the image of the irreducible Galois representation
$\Ind_{\Gal(\overline{\Q}/K)}^{\Gal(\overline{\Q}/\Q)}(\chi)$. This Galois representation
is odd, since complex conjugation plays the role of~$\sigma$ and consequently has
determinant~$-1$. Moreover, a well-known formula gives the conductor.
For more details, see~\cite{wiese}.
\end{proof}

If $M \subset N$ are two sets of natural numbers, then we say that
\emph{$M$ has density $\alpha = \natden MN$ in~$N$} if the limit for $x \to \infty$ of
\[
\frac{\#\{m \in M|\,m \le x\}}{\#\{m \in N|\,m \le x\}}
\]
exists, and is equal to~$\alpha$.

We will also introduce some notation for class groups. We denote by~$\CL(\Q(\sqrt{-d}))$ the class group of the imaginary quadratic field~$\Q(\sqrt{-d})$.

Let $p$ be a prime. We consider the following Hypothesis, which we denote by
$(H_p)$.
\begin{quote}
The density of the set
$$ \{ N \in \N \;|\; \exists\, d \in \N \textrm{ squarefree}, d \equiv 3 \mod 4, d \mid N,
 \CL(\Q(\sqrt{-d})) \textrm{ is $p$-suitable}\}$$
exists and is~$1$.
\end{quote}

We will establish $(H_2)$ by a result on the exponent of class groups.
Unfortunately, we do not know of any way of proving $(H_p)$ for odd~$p$, but
we shall show that $(H_p)$ is a consequence of the Cohen-Lenstra heuristics.

\begin{proposition}
Assume Hypothesis $(H_p)$. Then the conclusion of Theorem~\ref{main-theorem} is true.
\end{proposition}

\begin{proof}
Let $N \in \N$ such that there is a squarefree $d \equiv 3 \mod 4$ with $d \mid N$
for which $\CL(\Q(\sqrt{-d}))$ is $p$-suitable.
It suffices to show that there is a cuspidal Hecke eigenform $f$ modulo~$p$ of level~$N$,
quadratic Dirichlet character and some weight such that it has some coefficient
$a_n(f)$ in its $q$-expansion which does not lie in~$\F_p$.

Let us take such an~$N$ and~$d$. By means of Proposition~\ref{prop:rep} there is an
odd dihedral Galois representation $\rho: \Gal(\overline{\Q}/\Q) \to \GL_2(\overline{\F}_p)$
of conductor $d$ such that not all its traces lie in~$\F_p$.
By work of Hecke, $\rho$ is known to be modular of level~$d$, weight~$1$ for the
quadratic Dirichlet character belonging to $\Q(\sqrt{-d})$. This means that
there is a holomorphic Hecke eigenform~$f$ in the specified level, weight and character
whose coefficients of the standard $q$-expansion $a_\ell(f)$ at primes~$\ell$
reduce modulo a suitable prime lying above~$p$ to the trace of a Frobenius element
at~$\ell$. In particular, there is a prime~$\ell$ such that the Hecke polynomial at~$\ell$
is not completely split modulo~$p$. Noting that via the degeneracy maps $f$ gives
rise to a form in level~$N$ settles the claim for odd~$p$.

In order to treat the case $p=2$, we use the well-known fact that there is a congruence
modulo a prime above~$2$ of~$f$
and another modular form of the same level and the trivial Dirichlet character.
\end{proof}

\section{Exponents of class groups and proof of~$(H_2)$}\label{sec:htwo}

In this section, we will show Hypothesis $(H_2)$. Let us call a positive integer $d \equiv 3 \mod 4$
which is squarefree \emph{$2$-suitable} if $\CL(\Q(\sqrt{-d}))$ is $2$-suitable.
In order to show that the set of positive integers $N$ having a $2$-suitable squarefree positive divisor
$d \equiv 3 \mod 4$ has density~$1$, we first need to rule out the possibility that the only cyclic quotients are of order~$2^2-1=3$. To do this, we recall the following result on class groups with small exponent.

\begin{theorem}[Boyd-Kisilevksy~\cite{boyd-kisilevsky}, Weinberger~\cite{weinberger}]
There are only finitely many negative fundamental discriminants~$d$ such that $\CL (\Q(\sqrt{d}))$ has exponent~3.
\end{theorem} 
We note that this result is ineffective because it relies on Siegel's ineffective lower bound for the size of the class group. A computation is reported in~\cite{schuett} which says that the largest fundamental discriminant with absolute value less than~$10^6$ such that the exponent is~3 is~$-4027$; $\CL (\Q(\sqrt{-4027}))$ is isomorphic to~$C_3 \times C_3$; it is possible that this is the largest such fundamental discriminant.

We now note by genus theory that if~$d\equiv 3 \mod 4$ is a prime number, then the 2-part of the class group of~$\Q(\sqrt{-d})$ is trivial, so this means that for all but finitely many of these~$d$ both of the conditions of 2-suitability are satisfied. We will now use a well-known theorem of Landau to show that the set of natural numbers~$N$ which are divisible by such a~$d$ has density~1.

\begin{theorem}[Landau~\cite{landau}, pp 668--669]
Let $\{a_i\}$ be $r$ distinct residue classes modulo an integer~$A$ and let~$\mathcal{P}$ be the set of prime numbers which are congruent to one of the~$a_i$ modulo~$A$. If we let~$M(x)$ be the number of natural numbers less than~$x$ whose prime factors are all in $\mathcal{P}$, then
\[
M(x)\sim c \cdot \frac{x}{(\log x)^{1-\frac{r}{\phi(A)}}},
\]
where $c$ is a positive constant and $\phi$ is Euler's $\phi$-function.
Note that $\frac{r}{\phi(A)}$ is the Dirichlet density of the set~$\mathcal{P}$.
\end{theorem}
In particular, this means that the set of natural numbers whose prime factors are all congruent to a restricted set of the possible residue classes for a prime number modulo~$A$ has natural density~0, so the set of natural numbers with a prime factor which is a 2-suitable~$d$ has density~1, which is what we wanted to show. This means that we have proved the following proposition:
\begin{proposition}
The hypothesis~$(H_2)$ is true, and therefore the conclusion of Theorem~\ref{main-theorem} is true if~$p=2$.
\end{proposition}
The technique used to prove this hypothesis is likely to only give a density~0 result; we will now give some numerical data which suggests this.

Using a computer algebra package such as \Magma{}~\cite{magma} one finds that there are many quadratic 
imaginary fields whose class groups have trivial odd part; for instance, if we consider quadratic imaginary fields with fundamental discriminant of absolute value less than~$4,000,000$ there are~3722 
fields with class group of order~128, 8361 fields with class group of 
order~256, and~18046 fields with class group of order~512. This numerical evidence seems to suggest that there are an infinite number of imaginary quadratic fields with class number a power of~2, as one can find fields with class number a very high power of~2; for instance, it can be shown that~$\Q(\sqrt{-5000948753})$ has class number~65536. For more numerical results, see Section~10 of~\cite{rosen-silverman} which gives tables of the number of quadratic imaginary fields with small odd part with fundamental discriminant between~$-500000$ and~$-1000000$.

\section{Cohen-Lenstra heuristics and Hypothesis~$(H_p)$}\label{sec:hp}

In this section we want to make use of the Cohen-Lenstra heuristics
\cite{cohen-lenstra} for class groups of imaginary quadratic fields.
We first recall their principal definitions and their fundamental
heuristic assumption. We will, however, specialize them directly to
imaginary quadratic fields. We abbreviate the words {\em
fundamental discriminant} by {\em f.d.}

\begin{definition}[Cohen-Lenstra~\cite{cohen-lenstra},
Definition~5.1]\label{def:cl}
\begin{enumerate}[(a)]
\item Let $G$ be an abelian group. Define
$$w(G) = \left(\#\Aut(G)\right)^{-1}.$$
This will play the role of a weighting factor in the heuristics.

\item Let $\underline{A}$ be a set of isomorphism classes of abelian
groups and let $f$ be a complex-valued function on~$\underline{A}$.
The $(\infty,0,\underline{A})$-average of~$f$ is defined as
$$  M_{(\infty,0,\underline{A})}(f) := \lim_{x \to \infty}
\frac{\sum_{1 \le a \le x} \sum_{G \in \underline{A}, \#G=a} f(G) w(G)}
   {\sum_{1 \le a \le x} \sum_{G \in \underline{A}, \#G=a} w(G)},$$
if this limit exists.
\end{enumerate}
\end{definition}

We introduce some notation. For sake of shortness, we write $\CL(-d)$ for
$\CL(\Q(\sqrt{-d}))$. Let $S$ be a set of primes. We denote the
prime-to-$S$-part
of an abelian group $C$ by $\ppp{S}(C)$ and the $S$-part by $\pp{S}(C)$.
Moreover, if $S = \{p\}$, then we write $\ppp{p}(C)$ and $\pp{p}(C)$,
respectively.
In fact, we consider $\ppp{S}$ and $\pp{S}$ as functions on
(isomorphism classes of) finite abelian groups.
We write $\Ab{S}$ for the set of isomorphism classes of finite abelian groups of order
divisible only by primes in~$S$. Accordingly, we use the notation $\Abp{S}$ to
stand for the set of isomorphism classes of finite abelian groups of order
coprime with any prime in~$S$.

\begin{fha}[Cohen-Lenstra~\cite{cohen-lenstra}, Fundamental
Assumptions~8.1]\label{fha:cl}
Let $f$ be a function on~$\Abp{2}$ taking values in $\{0,1\}$.
Let
$$\cF := \{ d \in \N \;|\; -d \textrm{ is f.d.}\}$$
and define
$$\cF_f := \{ d \in \cF \;|\; f(\ppp{2}(\CL(-d))) = 1 \}.$$
Then the natural density $\natden{\cF_f}{\cF}$ of $\cF_f$ in~$\cF$ exists and
is equal to
$M_{(\infty,0,\Abp{2})}(f)$.
\end{fha}

\begin{proposition}\label{prop:cl}
Let $S$ be a set of primes containing the prime~$2$ which has a
density strictly smaller than~$1$. Assume the Cohen-Lenstra
heuristics, i.e.\ assume that Fundamental Heuristic
Assumption~\ref{fha:cl} is satisfied. Then the set
$$\cF_S := \{ d \in \cF \;|\; \# \ppp{S}(\CL(-d)) = 1 \}$$
has natural density~$0$ in the set~$\cF$.
\end{proposition}

\begin{proof}
Let $f$ be the function which sends the trivial abelian group to~$1$ and the
isomorphism class of any nontrivial abelian group to~$0$.
The principal input for this proof is \cite{cohen-lenstra},
Proposition~5.6, which implies that
$$ M_{\left(\infty,0,\Abp{2}\right)} (f \circ \ppp{S}) =
M_{\left(\infty,0,\Abp{S}\right)}
(f).$$
It follows directly from Definition~\ref{def:cl} that the right hand side term
is equal to
$$\lim_{x \to \infty}
\left(\sum_{a=1}^x \underset{\#G=a}{\sum_{G \in \Abp{S}}} w(G)\right)^{-1}.$$
This limit is~$0$ because the sum is larger than
$$ \underset{p \not\in S}{\sum_{p \textrm{ prime}}} \frac{1}{p-1},$$
which is divergent.
Under the Fundamental Heuristic Assumption~\ref{fha:cl},
the meaning of
$$M_{(\infty,0,\Abp{2})} (f \circ \ppp{S})$$
is the natural density of $\cF_S$ in~$\cF$.
\end{proof}

For a subset $A \subseteq \N$ we introduce the following shorthand notation:
$$ A(x) := \# \{ a \in A \;|\; a < x\}.$$
We first prove a simple lemma on the natural density~$\Delta(A,B)$.

\begin{lemma}\label{lem:natden}

\begin{enumerate}[(a)]
\item\label{it:a} Let $A \subseteq B \subseteq C$ be subsets of~$\N$
such that $\natden AB = \natden BC = 1$.
Then $\natden AC = 1$.

\item\label{it:b} Let $A \subseteq B \subseteq C$ be subsets of~$\N$
such that $\natden AC > 0$. Then $\natden AC = \natden BC > 0$
if and only if $\natden AB = 1$.

\item\label{it:c} Let $n$ be any positive integer and let $A \subseteq B$
 be subsets of~$\N$ such that $\natden AB = 1$.
Denote by $nA = \{ na \;|\; a \in A\}$ and similarly for $nB$.
Then $\natden{nA}{nB} = 1$.

\item\label{it:d} Let $A \subseteq B$ and $C \subseteq B$
be subsets such that $\natden AB = 1$ and $\natden{C}{B} = \alpha > 0$.
Then $\natden{A\cap C}{C} = 1$.

\item\label{it:e} Let $A_n \subseteq A_{n+1}$ for all~$n \in \N$ be subsets
of a set $B \subseteq \N$ all having a natural density $\natden{A_n}{B}$.
Assume that $\lim_{n\to \infty} \natden{A_n}{B} = 1$.
Let $A = \bigcup_{n\in \N} A_n$.
Then $\natden AB = 1$.

\item\label{it:f} Let $A_1 \subseteq B_1$ and $A_2 \subseteq B_2$ be subsets
of~$\N$
such that $\natden {A_i}{B_i} = 1$ for $i=1,2$.
Then $\natden{A_1 \cup A_2}{B_1 \cup B_2} = 1$.
\end{enumerate}
\end{lemma}

\begin{proof}
(\ref{it:a}), (\ref{it:b}) and (\ref{it:c}) are clear.

(\ref{it:d}) Choose $0 < \delta < \alpha$. There is a bound $D$ such that
$| \frac{C(x)}{B(x)} - \alpha | \le \delta$ for all $x \ge D$. This implies
$(\alpha -  \delta) B(x) \le C(x)$ for all $x \ge D$.
Note the trivial inequality $C(x) - (A\cap C)(x) \le B(x) - A(x)$, which
is valid for all~$x \in \N$.
Let now $\epsilon > 0$ be given. By assumption, there is a bound $E_\epsilon$
such that for all $x \ge E_\epsilon$ we have
$$ B(x) - A(x) \le \epsilon(\alpha-\delta) B(x).$$
Putting the inequalities together we obtain
$$C(x) - (A\cap C)(x) \le \epsilon C(x)$$
for all $x \ge \max( E_\epsilon, D)$ and thus the claim.

(\ref{it:e}) Put $a_{n,x} = \frac{A_n(x)}{B(x)}$. For all~$x$ and all~$n$ we
have
$a_{n+1,x} \ge a_{n,x}$. By assumption, for fixed~$n$ the limit
$\lim_{x \to \infty} a_{n,x} =: a_n$ exists,
we have $a_{n+1} \ge a_n$ for all~$n$ and $\lim_{n \to \infty} a_n = 1$.

Let $\epsilon>0$ be given. There exists a bound~$C$
such that for all $n \ge C$ we have $1 \ge a_n \ge 1 - \frac{1}{2} \epsilon$.
Moreover, there also exists a bound $D$ such that for all~$x \ge D$
we have $| a_{C,x} - a_C | \le \frac{1}{2} \epsilon$.
Hence, for all $x \ge D$ and all $n \ge C$ we have
$$ 1- \epsilon \le a_C-\frac{1}{2}\epsilon \le a_{C,x} \le a_{n,x} \le 1.$$
The claim follows.

(\ref{it:f}) Let $A=A_1 \cup A_2$ and $B=B_1 \cup B_2$.
Note the following inequality which is valid for all~$x \in \N$:
$$ B(x) - A(x) \le B_1(x)-A_1(x) + B_2(x) - A_2(x).$$
Let $\epsilon > 0$ be given. By assumption there exists a bound~$C$ such that
for all $x \ge C$ we have $B_i(x) - A_i(x) \le \frac{1}{2} \epsilon B_i(x)$ for
$i=1,2$.
Moreover, for all $x \in \N$ we have the inequality
$$ B(x) \ge \max ( B_1(x), B_2(x) ) \ge \frac{1}{2} (B_1(x) + B_2(x)).$$
Putting the inequalities together yields
$$  B(x) - A(x) \le \epsilon B(x)$$
for all $x \ge C$ and thus $\natden{A}{B} = 1$.
\end{proof}

The next lemma will be useful for deriving Hypothesis $(H_p)$ from the
Cohen-Lenstra heuristics.

\begin{lemma}\label{lem:one}
\begin{enumerate}[(a)]
\item \label{itd:a} Let $B$ be the set of positive integers that are $\equiv 3 \mod 4$
and let $A$ be the subset of those that are squarefree.
Let $A_N := \bigcup_{n=1}^N (2n-1)^2 A$.
Then $\lim_{N\to \infty} \natden{A_N}{B} = 1$.

\item \label{itd:b} Let $A$ be the set of positive integers that are $\equiv 3 \mod 4$
and $B$ the set of positive integers that are $\equiv 1 \mod 4$.
Let $C_N = B \setminus \bigcup_{i=1}^N p_i A$, where $p_1, p_2,\dots$ are the prime
numbers that are $\equiv 3 \mod 4$.
Then $\lim_{N\to \infty} \natden{C_N}{B} = 0$.

\item \label{itd:c} Let $A$ be the set of positive integers that are not divisible by~$4$.
Let $C_N = \N \setminus \bigcup_{n=0}^N 4^n A$. 
Then $\lim_{N\to \infty} \natden{C_N}{\N} = 0$.
\end{enumerate}
\end{lemma}

\begin{proof}
(\ref{itd:a}) The set $A_N$ is the subset of~$B$ of those integers that are divisible
by some odd square $(2n-1)^2$ for $n \le N$. It is well-known that $\natden AB = \frac{8}{\pi^2}$.
Hence,
$$ \natden{A_N}{B} = \natden{A}{B} \sum_{n=1}^N \frac{1}{(2n-1)^2}
\xrightarrow{N \to \infty} \frac{8}{\pi^2}\frac{\pi^2}{8} = 1,$$
since $n^2 A \cap m^2 A = \emptyset$ for $n \neq m$. From this the claim follows.

(\ref{itd:b}) The set $C_N$ is contained in the set of positive integers~$n$
that are $\equiv 1 \mod 4$ and are not divisible by any of $p_1,p_2,\dots, p_N$.
The latter condition can be reformulated to say that $n$ is a unit in
$\Z/(p_1\cdot p_2 \cdot \ldots \cdot p_N) \Z$. Thus, the density of $C_N$ in~$B$ is
$$ \prod_{i=1}^N \frac{p_i-1}{p_i} = \prod_{i=1}^N \left(1-\frac{1}{p_i}\right) =
\left( \prod_{i=1}^N \left(1-\frac{1}{p_i}\right)^{-1} \right)^{-1} =
\left( \underset{\textrm{cond}_N}{\sum_{n=1}^\infty} \frac{1}{n} \right)^{-1}$$
with the condition $\textrm{cond}_N$ that $n$ is only divisible by the primes $p_1,p_2,\dots, p_N$.
It is well-known that the sequence $a_N := \sum_{n=1, \textrm{ cond}_N}^\infty \frac{1}{n}$ diverges, whence the claim follows.

(\ref{itd:c}) The set $C_N$ is the set of positive integers divisible by $4^{N+1}$,
which obviously has density $1/4^{N+1}$, whence the claim.
\end{proof}

\begin{proposition}
Let $p$ be an odd prime.
Assume the Cohen-Lenstra heuristics, i.e.\ Fundamental Heuristic
Assumption~\ref{fha:cl}.
Then Hypothesis $(H_p)$ is satisfied, and therefore the conclusion of Theorem~\ref{main-theorem} is true
for odd primes.
\end{proposition}

\begin{proof}
Let $S$ be the set of primes dividing $p(p^2-1)$. Let $\cF$ and $\cF_S$
be the sets in Proposition~\ref{prop:cl}.
Let $A = \cF \setminus \cF_S$.

The strategy of the proof is to extend the fact that $A$ has natural density~$1$ in~$\cF$
to all positive integers, by first extending it to the integers $\equiv 3 \mod 4$,
then to those $\equiv 1 \mod 4$ and finally by multiplying by $2$ and powers of~$4$
to all integers.

Define
$$ B_1 := \{d \in \N \;|\; d \equiv 3 \mod 4, d \textrm{ squarefree }\} \subset \cF$$
and $A_1 := B_1 \cap A$.
By Lemma~\ref{lem:natden}~(\ref{it:d}), we have $\natden{A_1}{B_1} = 1$.
Due to the statement on finite unions, Lemma~\ref{lem:natden}~(\ref{it:f}),
we have
$$ \natden{\bigcup_{n=1}^N (2n-1)^2 A_1}{\bigcup_{n=1}^N (2n-1)^2 B_1}=1.$$
Furthermore,
$$ B_2 := \{ n \in \N \;|\; n \equiv 3 \mod 4\} = \bigcup_{n=1}^\infty (2n-1)^2 B_1$$
and by Lemma~\ref{lem:one}~(\ref{itd:a}) also
$$ \lim_{N \to \infty} \natden{\bigcup_{n=1}^N (2n-1)^2 B_1}{B_2} = 1.$$
By Lemma~\ref{lem:natden}~(\ref{it:b}) this immediately implies
$$ \lim_{N \to \infty} \natden{\bigcup_{n=1}^N (2n-1)^2 A_1}{B_2} = 1,$$
whence Lemma~\ref{lem:natden}~(\ref{it:e}) yields
$$ \natden{A_2}{B_2} = 1 \textrm{ with } A_2 = \bigcup_{n=1}^N (2n-1)^2 A_1.$$
We have achieved the first goal, namely, to extend the density one statement
to the integers that are $\equiv 3 \mod 4$.

Next, we multiply the sets $A_2$ and $B_2$ by all the primes $\equiv 3 \mod 4$
in order to pass to the integers that are $\equiv 1 \mod 4$.
Let $B_3 := \{n \in \N \;|\; n \equiv 1 \mod 4\}$. From Lemma~\ref{lem:one}~(\ref{itd:b})
we obtain
$$ \lim_{N \to \infty} \natden{\bigcup_{i=1}^N p_i B_2}{B_3}= 1,$$
whence we get (using Lemma~\ref{lem:natden} (\ref{it:b}) and~(\ref{it:f}))
$$ \lim_{N \to \infty} \natden{\bigcup_{i=1}^N p_i A_2}{B_3}= 1.$$
Setting $A_3 := \bigcup_{i=1}^\infty p_i A_2$,
we obtain from Lemma~\ref{lem:natden}~(\ref{it:e})
$$ \natden{A_3}{B_3}=1.$$

Now we have also achieved our goal for the integers that are $\equiv 1 \mod 4$.
We get those $\equiv 2 \mod 4$ simply by multiplying those that we have treated
so far by~$2$.
Let
$$ A_4 := A_2 \cup A_3 \textrm{ and } B_4 := B_2 \cup B_3.$$
By Lemma~\ref{lem:natden}~(\ref{it:f}), it follows that $\natden{A_4}{B_4}=1$.
Moreover, the density of $A_5 := 2A_4$ in $B_5 := 2 B_4$ is~$1$
by Lemma~\ref{lem:natden}~(\ref{it:a}).

Finally, using the same procedure and Lemma~\ref{lem:one}~(\ref{itd:c}), we find
$$ \natden{A_6}{\N} = 1 \textrm{ with } A_6 := \bigcup_{n=0}^\infty 4^n A_5.$$
Now we note that the set $A_6$ is precisely the set of $n \in \N$
such that there is a squarefree $d \in \N$ such that $d \mid n$ and $d \equiv 3 \mod 4$
with the property that $d \not\in \cF_S$.
Directly from the definition of $p$-suitability it follows that such a~$d$ is $p$-suitable.
\end{proof}

\begin{remark}
We remark that a straightforward generalization of Corollary~3 of~\cite{sound} 
yields that for a given prime~$p$ the set of $d \in \N$ such that $-d$ is a 
fundamental discriminant and the class group of $\Q(\sqrt{-d})$ is a 
$p$-group has density zero in~$\N$.

However, we were unable to extend this result to groups whose orders are only divisible
by primes in some finite set~$S$. This would have sufficed to prove Hypothesis 
$(H_p)$ without assuming the Cohen-Lenstra heuristics.
\end{remark}

\section{Acknowledgments}

This project was started while the second author was visiting the University of Bristol; he would like to thank the University and the Heilbronn Institute for their hospitality. The work was finished while the first author was visiting the Institut f\"ur Experimentelle Mathematik (IEM). He would like to
thank the IEM and Universit\"at Duisburg-Essen for their hospitality.

G.~W.\ was partially supported by the European Research Training Network
{\em Galois Theory and Explicit Methods} MRTN-CT-2006-035495
and by the Sonderforschungsbereich Transregio 45 of the Deutsche Forschungsgemeinschaft.

\vspace*{.2cm}

\noindent L.~J.~P.~Kilford\\
Department of Mathematics\\
University Walk\\
Bristol\\
BS8 1TW\\
United Kingdom\\
E-mail: \url{l.kilford@gmail.com}\\
Web page: \url{http://www.maths.bris.ac.uk/~maljpk/}

\bigskip

\noindent Gabor Wiese\\
Universit\"at Duisburg-Essen\\
Institut f\"ur Experimentelle Mathematik\\
Ellernstra{\ss}e 29\\
45326 Essen\\
Germany\\
E-mail: \url{ gabor.wiese@uni-due.de}\\
Web page: \url{ http://maths.pratum.net/}

\end{document}